\documentclass{article}
\usepackage{hyperref}
\usepackage[utf8]{inputenc}
\usepackage{listings}
\usepackage[bottom=1in, top=1in, left=1.25in, right=1.25in]{geometry}
\usepackage{todonotes}
\usepackage{xcolor}
\usepackage{comment}
\usepackage{graphicx}
\usepackage{url}
\usepackage{tikz}
\usepackage{mathtools}
\usepackage{amsthm}
\usetikzlibrary{positioning, calc}

\newtheorem{thm}{Theorem}[section]
\newtheorem{lemma}[thm]{Lemma}

\newtheorem{conj}[thm]{Conjecture}

\newtheorem{definition}[thm]{Definition}

\newcommand{\iso}{\mathrm{iso}}

\usepackage{listings}
\usepackage{amsfonts}
\usepackage{color}
\usepackage{amsmath}

\definecolor{dkgreen}{rgb}{0,0.6,0}
\definecolor{gray}{rgb}{0.5,0.5,0.5}
\definecolor{mauve}{rgb}{0.58,0,0.82}

\title{Number of Independent Sets in Regular and Irregular Graphs: A 31 Year Journey}
\date{\today}

\author{Dev Chheda, Ram Goel, Eddie Qiao}

\begin{document}
\maketitle

\begin{abstract}
We review the progress made on bounding the number of independent sets in $d$-regular and irregular graphs over the last 31 years. We particularly focus on contributions from Kahn \cite{KAHN_2001}, Zhao \cite{ZHAO_2009}, and Sah et al. \cite{sah2019} in incrementally proving stronger and more general versions of the upper bound. We reproduce the main results of these works, particularly focusing on the unweighted special case (with fugacity $\lambda = 1$), which allows us to provide more intuitive and clear explanations of the key ideas that have been developed in the field over three decades. 
\end{abstract}

\section{Introduction}
\label{Intro}

Let $G$ be a graph on $n$ vertices. Denote the collection of independents sets of $G$ by $\mathcal{I}(G)$, and let $i(G) = |\mathcal{I}(G)|$ be the number of independent sets. What is the maximum value of $i(G)$ when $G$ has $n$ vertices, and which graphs achieve this maximum? This paper explores the key ideas, developed over more than three decades, which resulted in the eventual resolution of this question by Sah et al. in 2019 \cite{sah2019}. 

\subsection{Historical Context}
We begin with some historical context regarding this problem (also see \cite{galvin2009upper} and \cite{ZHAO_2009}). In the 1988 Number Theory Conference at Banff \cite{banff}, Granville conjectured that for $d$-regular graphs, $i(G) \leq 2^{(1/2 + o(1))n}$. This conjecture was resolved in the 1991 paper from Alon \cite{alon1991independent}, who showed that $i(G) \leq 2^{(1/2 + O(d^{-0.1}))n}$. Alon additionally conjectured that $i(G)$ is maximized when $G$ is a disjoint union of $n/2d$ complete bipartite graphs $K_{d,d}$.  This conjecture was formalized by Kahn \cite{KAHN_2001} in 2001:
\begin{conj}[Alon \cite{alon1991independent} and Kahn \cite{KAHN_2001}]
\label{conj:regular}
    Let $G$ be a $d$-regular graph with $n$ vertices. Then,
    \begin{align*}
        i(G) \leq i(K_{d,d})^{n/(2d)} = (2^{d+1} - 1)^{n/(2d)}.
    \end{align*}
\end{conj}

Kahn \cite{KAHN_2001} proved that \ref{conj:regular} result holds when $G$ is bipartite by using entropy methods. The next major improvement came in 2009 from Zhao \cite{ZHAO_2009} who used a bipartite-swapping trick to extend Kahn's result to all $d$-regular graphs $G$, thereby resolving Conjecture \ref{conj:regular}. 

Kahn also conjectured an extension to irregular graphs, which was finally resolved by Sah et al. in 2019 \cite{sah2019}. Specifically, the result proven by Sah et al. is:
\begin{thm}
\label{thm:main}
    Let $G$ be a graph without isolated vertices. Let $d_v$ be the degree of $v$ in $G$. Then,
    \[ i(G) \le \prod_{uv\in E(G)} i(K_{d_u,d_v})^{1/(d_ud_v)}. \]
    Equality holds if and only if $G$ is a disjoint union of complete bipartite graphs. 
\end{thm}

\subsection{Paper Overview}
In our paper, we first cover the main results from Kahn \cite{KAHN_2001} in Section \ref{Kahn} and Zhao \cite{ZHAO_2009} in Section \ref{Zhao} which prove Conjecture \ref{conj:regular}, which is essentially a specialized form of Theorem \ref{thm:main} for $d$-regular graphs. We then cover the key methods and results from Sah et al. \cite{sah2019} in Section \ref{Sah}, of which the primary result is Theorem \ref{thm:main}. Sah et al. also prove some lower bounds on $i(G)$, but we only focus on covering the upper bound. We briefly state the lower bound results and some generalizations of the upper bound results in Section \ref{further}. 

All of the results which we cover can be extended using the hard-core model. The partition function of the hard-core model for a graph $G$ is $P_G(\lambda) = \sum_{I \in \mathcal{I}(G)} \lambda^{|I|}$ where the parameter $\lambda$ is often referred to as \textit{fugacity}. Each of the key results which we cover in Sections \ref{Kahn} - \ref{Sah} can be generalized with a weighted version using the hard-core model. 

For example, a weighted version of Theorem \ref{thm:main} using the partition function is:
\begin{thm}
\label{thm:main-general}
    Let $G$ be a graph without isolated vertices. Let $d_v$ be the degree of $v$ in $G$. Let $\lambda > 0$. Then,
    \[ P_G(\lambda) \leq \prod_{uv\in E(G)} P_{K_{d_u,d_v}}(\lambda)^{1/(d_ud_v)}. \]
\end{thm}
Sah et al. \cite{sah2019} prove this more general version (along with a double-weighted version which we discuss in Section \ref{subsec:general}). 

We only focus on the $\lambda = 1$ case (e.g. which would reduce Theorem \ref{thm:main-general} to Theorem \ref{thm:main}), since this allows us to focus the key ideas, methods, and insights, all of which remain the same even in the more general case.

\section{Bipartite and Regular Graphs -- Kahn}
\label{Kahn}
\subsection{Overview of Kahn's Result}
Kahn proved the first result in this progression in 1999, by proving Theorem \ref{thm:main} for the case of $G$ being bipartite and $d$-regular. 
\begin{thm}[Theorem 1.9 of Kahn \cite{KAHN_2001}]
    \label{thm:Kahn}
    If $G$ is a $d$-regular bipartite graph on $n$ vertices, then
    \[ i(G) \le (2^{d+1}-1)^{n/(2d)}. \]
    Equality holds when $G$ is a disjoint union of $K_{d,d}$'s. 
\end{thm}
In Sah et al. \ref{Sah}, the upper bound above is stated also as $i(K_{d,d})^{n/(2d)}$, since $i(K_{d,d}) = 2^{d+1}-1$. This is because each independent set of a bipartite graph is a subset of one of the two parts, and each part has $2^d$ subsets. We subtract 1 for overcounting the empty set twice.

In Kahn \cite{KAHN_2001}, Theorem \ref{thm:Kahn} is stated in log-form:
\[ \log i(G) \le \frac{n}{2d} \log(2^{d+1}-1). \]
The reason for this is that the argument used to prove this is based on entropy, which is more conducive to the log-form. We will go into detail on relevant background and notation in \ref{subsection:Entropy}, and the proof in \ref{subsection:Kahn_Proof}. 

\subsection{Entropy}
\label{subsection:Entropy}
Entropy is a fundamental concept in information theory, originally introduced by Claude Shannon. It quantifies the uncertainty or the average amount of information inherent in a random variable's possible outcomes. The entropy of a discrete random variable $X$ with pdf $p(x) = \text{Pr}(X=x)$ is defined as 
\[ H(X) = \sum_x p(x) \log \frac{1}{p(x)}. \]
The base of the logarithm is set to 2, in which case the entropy is measured in bits.

Conditional entropy is given by
\[ H(X\mid Y) = \mathbb{E}H(X\mid Y=y) = \sum_y p(y) \sum_x p(x\mid y) \log \frac{1}{p(x\mid y)}. \]
For a random vector $X = (X_1,\ldots,X_n)$, define
\[ H(X) = \sum_{i=1}^n H(X_i \mid X_{i-1},\ldots,X_1). \]
Some useful inequalities are used. These are well-known and often used in literature on information theory and are used freely without proof. 
\begin{equation}
    H(X) \le \log |\text{range}(X)|.
\end{equation}
\begin{equation}
        H(X\mid Y) \le H(X). 
\end{equation}
The above equation essentially says that given more information, entropy decreases. This intuition is generalized by the fact that if $Y$ determines $Z$, then $H(X\mid Y) \le H(X\mid Z)$. Similarly the joint entropy of a collection of r.v.'s is at most the sum of the individual entropies, which is intuitively true since there could be correlation among the r.v.'s. A more general version of this fact is the following:

\begin{lemma}[Shearer's Lemma]
    \label{lemma: Shearer}
    Let $X=(X_1,\ldots,X_n)$ be a random vector and $\mathcal{A}$ a collection of subsets of $[n]$, with each element of $[n]$ contained in at least $m$ members of $\mathcal{A}$. Then
    \[ H(X) \le \frac1m \sum_{A\in \mathcal{A}} H(X_A). \]
\end{lemma}
The notation $X_A$ denotes the restriction of $X$ onto indices $A$, i.e. $X_A = (X_i : i\in A)$. 

\subsection{Proof of Theorem \ref{thm:Kahn}}
\label{subsection:Kahn_Proof}

\subsubsection{Notation}
We first set up notation. Let $I$ be a random member of $\mathcal{I}(G)$, chosen uniformly at random. This is one of the key ideas, since a uniform distribution makes 
\[ H(I) = \sum_{I\in \mathcal{I}(G)} \frac{1}{i(G)} \log i(G) = \log i(G). \]
Hence we are interested in upper bounds on the entropy of $I$, which intuitively quantifies the information on $I$. 

Denote $\mathbf{1}_v$ to be the indicator of $v\in I$, and $p(v)$ for $\text{Pr}(v\in I)$. Let 
\[ X_v := I\cap N(v).\]
A key idea is that $I$ cannot contain $v$ and its neighbors. So only one of $\mathbf{1}_v$ and $X_v$ can contribute towards the information provided by $I$, which is in turn encoded by the entropy of $I$.

Denote 
\[ q(v) := \text{Pr}(Q_v)\] where the event $Q_v$ is defined as
\[ Q_v := \{X_v = \emptyset\} = \{ I \cap N(v) = \emptyset\}.\] That is, $Q_v$ is the event that no neighbors of $v$ are in the independent set $I$. 
Similar intuition as the above implies $X_v$ and $Q_v$ are uncorrelated in a sense, and $\mathbf{1}_v$ and $\overline{Q_v}$ are uncorrelated in a sense, the reason being that regardless of whether $v$ is in $I$ or not, if $X_v = \emptyset$, then the neighbors of $v$ provide no information on $I$. 

\subsubsection{Proof}
\label{prelim_ineqs}
Recall that $G$ is bipartite. Let the two parts be $\mathcal{E}$ and $\mathcal{O}$ (for even and odd, respectively). All sums will be over $v\in \mathcal{E}$. Note
\begin{align*}
    H(I) &= H(I\cap \mathcal{O}) + H(I\cap \mathcal{E} \mid I\cap \mathcal{O}) \\
    &\le H(I\cap \mathcal{O}) + \sum H(\mathbf{1}_v\mid I\cap \mathcal{O}) \\
    &\le H(I\cap \mathcal{O}) + \sum  H(\mathbf{1}_v\mid \mathbf{1}_{Q_v}) \\ 
    &= H(I\cap \mathcal{O}) + \sum  q(v) 
\end{align*}
The first equality follows since $G = \mathcal{E} \sqcup \mathcal{O}$. The second inequality follows by splitting up joint entropy into individual sums, and the last by containment of neighbors in the other bipartite part. The latter aforementioned step is where the bipartite part is crucial. The last equality follows since $H(\mathbf{1}_v\mid \mathbf{1}_{Q_v}) = q(v)$ due to picking $I$ uniformly at random. 

We establish one more inequality chain. We utilize Shearer's Lemma, noting that the sets $N(v)$ cover each $x \in \mathcal{O}$ exactly $d$ times, due to regularity. Each of the $d$ neighbors lies in the other bipartite part. Note
\begin{align*}
    H(I\cap \mathcal{O}) &\le \frac{1}{d} H(X_v) \\
    &= \frac1d \sum H(\mathbf{1}_{Q_v}) + H(X_v \mid \mathbf{1}_{Q_v}) \\
    &= \frac1d \sum H(q(v))  + (1-q(v))H(X_v\mid \overline{Q_v}) \\
    &\le \frac1d \sum H(q(v))  + (1-q(v))\log(2^d-1).
\end{align*}
The first step followed from Shearer's Lemma. The final inequality above follows from the crude bound $H(X) \le \log |\text{range}(X)|$, since there are $2^d-1$ nonempty subsets of $N(v)$ and $X_v = I\cap N(v) \subseteq N(v)$. 

Combining these two inequality chains, we get
\begin{align*}
    H(I) &\le \frac1d \sum H(q(v)) + (1-q(v)) \log (2^d-1) + \sum q(v) \\
    &= \frac{n}{2d} \log(2^d-1)  + \frac1d \sum H(q(v)) + q(v) \log \frac{2^d}{2^d-1}. 
\end{align*}
By differentiation, the maximum of the function $H(x) + x \log \frac{2^d}{2^d-1}$ occurs at $x_0 = \frac{2^d}{2^{d+1}-1}$. Now, plugging this back into the expression above and reducing algebraically gives the desired result of $\frac{n}{2d} \log(2^{d+1}-1)$, as desired. 

\section{Extension to all Regular Graphs -- Zhao}
\label{Zhao}

Kahn conjectured that the upper bound could be extended to all $d$-regular graphs, not just bipartite. In 2009, Zhao~\cite{ZHAO_2009} used a bipartite-swapping trick (later expanded on in \cite{zhao2011bipartite}) to resolve this conjecture. Specifically, Zhao proves the following:

\begin{thm}[Theorem 2 of Zhao \cite{ZHAO_2009}]
    \label{thm:regular}
    If $G$ is a $d$-regular graph on $n$ vertices, 
    \begin{align*}
        i(G) \leq i(K_{d,d})^{n/(2d)} = (2^{d+1}-1)^{n/(2d)}.
    \end{align*}
    Equality holds if and only if $G$ is a disjoint union of $K_{d,d}$’s.
\end{thm}

\subsection{Bipartite swapping trick}
\label{subsec:bipartite}

For $A, B \in V(G)$, we say that $A$ and $B$ are independent from each other if $G$ does not contain edge $ab$ for any $a \in A$, $b \in B$. Let $G[A]$ be the subgraph of $G$ induced by $A$. We define $\mathcal{J}(G)$ as the set of pairs $(A, B)$ such that $A$ and $B$ are independent and $G[A \cup B]$ is bipartite. 

\begin{lemma}[Bipartite-swapping trick of Zhao \cite{ZHAO_2009}]
\label{lma:bipartite}
     Let $G$ be any graph. There exists a bijection between $\mathcal{I}(G) \times \mathcal{I}(G)$ and $\mathcal{J}(G)$.
\end{lemma} 
\begin{proof}
    
For any $W \subset V(G)$, such that $G[W]$ is bipartite, let $W_1, W_2$ be the bipartition of $W$. Now, for $A, B \in \mathcal{I}(G)$, note that $G[A \cup B]$ is bipartite since $(A, B \backslash A)$ is a valid partition (due to independence). 

Letting $W = A \cup B$, we define $S_1 = A \cap W_1$, $S_2 = A \cap W_2$, $S_3 = B \cap W_1$, and $S_4 = B \cap W_2$, as shown in Figure \ref{fig:bipartite}. The lines in the figure represent the existence of edges between subsets of vertices. Since $W_1, W_2$ is a bipartition, there are no edges within $W_1$ or $W_2$, and, since $A, B$ are independent, there are no edges within $A$ or $B$. This means that the only possible edges are between $S_1$ and $S_4$ or $S_2$ and $S_3$, as shown. 

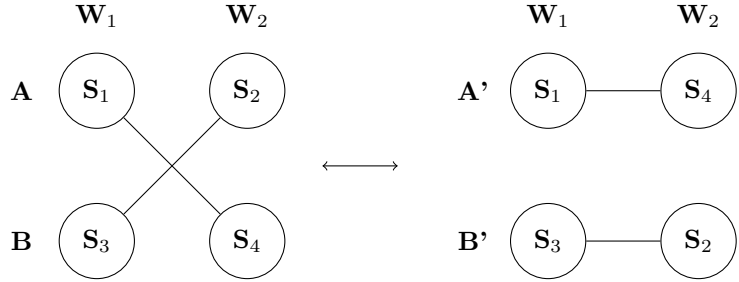
\begin{figure}[ht]
    \centering
    
\begin{tikzpicture}
    \node at (-1, 0) {\textbf{A}};
    \node at (-1, -2) {\textbf{B}};

    \node at (0, 1) {\textbf{W$_1$}};
    \node at (2, 1) {\textbf{W$_2$}};

    \node[circle, draw, minimum size=1cm] (s1) at (0, 0) {\textbf{S$_1$}};
    \node[circle, draw, minimum size=1cm] (s2) at (2, 0) {\textbf{S$_2$}};
    \node[circle, draw, minimum size=1cm] (s3) at (0, -2) {\textbf{S$_3$}};
    \node[circle, draw, minimum size=1cm] (s4) at (2, -2) {\textbf{S$_4$}};

    \draw[-] (s1) -- (s4);
    \draw[-] (s2) -- (s3);

    \draw[<->] (3, -1) -- (4, -1);
    
    \node at (5, 0) {\textbf{A'}};
    \node at (5, -2) {\textbf{B'}};

    \node at (6, 1) {\textbf{W$_1$}};
    \node at (8, 1) {\textbf{W$_2$}};

    \node[circle, draw, minimum size=1cm] (s1) at (6, 0) {\textbf{S$_1$}};
    \node[circle, draw, minimum size=1cm] (s4) at (8, 0) {\textbf{S$_4$}};
    \node[circle, draw, minimum size=1cm] (s3) at (6, -2) {\textbf{S$_3$}};
    \node[circle, draw, minimum size=1cm] (s2) at (8, -2) {\textbf{S$_2$}};

    \draw[-] (s1) -- (s4);
    \draw[-] (s3) -- (s2);
\end{tikzpicture}

    \caption{Bijection between $\mathcal{I}(G) \times \mathcal{I}(G)$ and $\mathcal{J}(G)$}
    \label{fig:bipartite}
\end{figure}

As shown in Figure \ref{fig:bipartite}, we will map $(A, B)$ to $(A', B') := (S_1 \cup S_4, S_2 \cup S_3) $. We can see that $(A', B') \in \mathcal{J}(G)$ since $A' \cup B' = A \cup B = W$ is clearly bipartite, and $A', B'$ are clearly independent from each other. 

Now, consider the inverse mapping which starts with $(A', B') \in \mathcal{J}(G)$. $W = A' \cup B'$ is bipartite by definition, and so we let $S_1 = A' \cap W_1$, $S_4 = A' \cap W_2$, $S_3 = B' \cap W_1$, and $S_2 = B' \cap W_2$ as shown in Figure \ref{fig:bipartite}. Similar to before, since $W$ is bipartite and $A', B'$ are independent from each other, the only edges in $G[W]$ are those between $S_1$ and $S_4$ and those between $S_2$ and $S_3$. So, if we map $(A', B')$ to $(A, B) = (S_1 \cup S_2, S_3 \cup S_4)$, we see that $A, B \in \mathcal{I}(G)$. 

So, we have clearly shown a bijection between $\mathcal{I}(G) \times \mathcal{I}(G)$ and $\mathcal{J}(G)$.
\end{proof}

\subsection{Proof of Theorem \ref{thm:regular}}

Now, let us see why this swapping trick extends the upper bound to non-bipartite graph. We will consider $G \times K_2$, the bipartite double cover of $G$, which contains vertices $(v, i)$ for $v \in V(G)$ and $i \in \{ 0, 1\}$ and edges $(u, 0), (v, 1)$ if and only if $uv \in E(G)$. Then, any independent set in $G\times K_2$ can be separated into its $i=0$ side $(A, 0)$ and $i=1$ side $(B, 1)$ where $A, B \subset V(G)$. Note that $(A, 0) \cup (B, 1) \in \mathcal{I}(G\times K_2)$ holds precisely when $A$ and $B$ are independent from each other in $G$. Then, we can see
\begin{align*}
    i(G \times K_2) &= |\{ A, B \subset V(G) : A \text{ independent from } B \}| \\
    &\geq  |\mathcal{J}(G)|\\
    &=  |\mathcal{I}(G) \times \mathcal{I}(G) | \\
    &= i(G)^2
\end{align*}
The second inequality is since $\mathcal{J}(G)$ is a subset of $\{ A, B \subset V(G) : A \text{ independent from } B \}$ since it has the additional bipartite condition. The third equality is from Lemma \ref{lma:bipartite}, since a bijection between $\mathcal{J}(G)$ and $\mathcal{I}(G) \times \mathcal{I}(G)$ means they have the same size. 

Note that if $G$ is $d$-regular, then so is $G \times K_2$ since degrees don't change. Since $G \times K_2$ is bipartite and has $2n$ vertices, Theorem \ref{thm:Kahn} gives 
\begin{align*}
    i(G \times K_2) \leq i(K_{d,d})^{(2n)/(2d)} = i(K_{d,d})^{n/d}
\end{align*}
and combining with the previous inequality gives
\begin{align*}
    i(G) \leq i(G \times K_2)^{1/2} \leq i(K_{d,d})^{n/(2d)}
\end{align*}
as desired. 

\section{Generalization to Irregular Graphs -- Sah et al.}
\label{Sah}

Despite the previous progress from Section~\ref{Zhao}, the graphs $G$ that we deal with are all $d$-regular. 
Sah~\cite{sah2019} generalizes this result to deal with all graphs.
To do so, techniques from the literature, particularly Section~\ref{Zhao}, are key to this generalization. 
However, the entropy approach taken in Section~\cite{KAHN_2001} is not sufficient, and Sah et al. find another approach to the problem.
Specifically, they show the following theorem.
\begin{thm}[Sah, Sawhney, Stoner, and Zhao~\cite{sah2019}]
\label{thm:sah-main}
Let $G$ be a graph, and let $d_v$ denote the degree of vertex $v$ in $G$. Then
\begin{align*}
i(G) &\leq 2^{\iso(G)} \prod_{uv \in E(G)} i(K_{d_u, d_v})^{1/(d_u d_v)} \\
&= 2^{\iso(G)} \prod_{uv \in E(G)} (2^{d_u} + 2^{d_v} - 1)^{1/(d_u d_v)}.
\end{align*}
Equality holds if and only if $G$ is a disjoint union of complete bipartite graphs and isolated vertices.
\end{thm}
In the paper, Sah et al. state the theorem for graphs $G$ without any isolated vertices (i.e. $\iso(G) = 0$). 
However, it is easy to see that proving this theorem is equivalent. 
Since each of the isolated vertices have a binary choice of whether or not they are in an independent set, they contribute a factor of $2^{\iso(G)}$ to $i(G)$.
The paper from Sah et. al. also differs because they provide a proof of a generalized version of this problem involving the hard-core model, which will be discussed in Section~\ref{further}. 
We will now provide a sketch of the proof, specifically proving Theorem~\ref{thm:sah-main} rather than proving the generalized version.
The hope is that the proof of this theorem rather than the generalized version will yield more intuition about what the inequalities and techniques mean.
We will take the following steps to solve the problem.
\begin{enumerate}
    \item Reduce the problem to assuming $G$ is bipartite using the bipartite swapping trick covered in Section~\ref{subsec:bipartite}.
    \item Use a recursion involving $i(G)$ to reduce the problem to another inequality.
    \item Apply Hölder's inequality and computation to prove this inequality.
\end{enumerate}
We first reduce the problem to assuming $G$ is bipartite.
First, define $j(G)$ to be 
\[j(G) \coloneqq 2^{\iso(G)} \prod_{uv \in E(G)} i(K_{d_u, d_v})^{1/(d_u d_v)}.\]
We claim that $j(G)^2 = j(G \times K_2)$. 
Recall that for each edge $(u, v) \in G$, we have the edges $((u, 0), (v, 1)), ((u, 1), (v, 0)) \in G \times K_2$.
This means that the degree of $(u, i)$ in $G \times K_2$ is equal to that of $u$ in $G$ for $i = 0, 1$. 
Thus, when computing the product for $j(G \times K_2)$, each of the edges within $G$ is counted twice, giving $j(G)^2 = j(G \times K_2)$.

Recall from the bipartite swapping trick in Lemma~\ref{lma:bipartite} that $i(G)^2 \le i(G \times K_2)$. 
Combining these two inequalities, in order to prove $i(G) \le j(G)$, it suffices to show that $i(G \times K_2) \le j(G \times K_2)$.
As a result, it suffices to show that $i(G) \le j(G)$ for all bipartite graphs $G$. 

Next, we claim that we can assume $G$ is connected. 
Observe that if $G$ can be divided into two connected components $G_1$ and $G_2$, then $i(G) = i(G_1) i(G_2)$ and $j(G) = j(G_1)j(G_2)$. 
Thus, the problem reduces down to proving the inequality for each of the connected components, so we can assume $G$ is connected. 
The proof so far has been the same technique used by Zhao to solve the problem for all $d$-regular graphs.
However, the proof differs from here because we no longer have the $d$-regular assumption. 

Now, we induct on the number of vertices in $G$. 
It is simple to check the base case, which is when $G$ is an empty graph or an isolated vertex. 
Note that for any vertex $w$, we have 
\[i(G) = i(G-w) + i(G-w-N(w)),\]
where $G-w$ denotes $G$ with the vertex $w$ and the edges with $w$ deleted, and $G-w-N(w)$ is the same with the neighbors of $w$ deleted as well. 
In this equation, $i(G-w)$ represent the independent sets that don't contain $w$, and $i(G-w-N(w))$ contain the independent sets that contain $w$.
Since any vertex $w$ works, we select $w$ to be the one with maximum degree. 
By using induction, we have 
\begin{align*}
    i(G-w) &\le j(G-w) \\
    i(G-w-N(w)) &\le j(G-w-N(w)).
\end{align*}
Combining these, we have 
\[i(G) = i(G-w) + i(G-w-N(w)) \le j(G-w) + j(G-w-N(w)).\]
Thus, it suffices to prove that 
\begin{align}
    j(G-w) + j(G-w-N(w)) \le j(G). \label{eq:j-inequality}
\end{align}
To tackle this inequality, define $V_k$ to be the set of vertices with distance exactly $k$ to $w$. 
Intuitively, this works well with the inequality we have because 
\begin{align*}
    G &= V_0 \cup V_1 \cup \dots \\
    G-w &= V_1 \cup V_2 \cup \dots \\
    G-w-N(w) &= V_2 \cup V_3 \cup \dots 
\end{align*}
With nice expressions for each of these three values, there is a lot of potential to prove (\ref{eq:j-inequality}).
Let $E_k$ be the edges between $V_{k-1}$ and $V_k$. 
Define $E_{\ge k} \coloneqq \bigcup_{i \ge k} E_i$ and $E_{\le k}$ analogously. 
Furthermore, for each $i = 1, 2, \dots$, we let $I_i$ be the vertices that are not connected to any vertices in $V_{i+1}$. 
Finally, let $\Delta$ be the degree of $w$.
These definitions can be visualized in Figure~\ref{fig:sah}. 
From this visualization, it is evident why it is important for $G$ to be bipartite.
\begin{figure}
    \centering 
    \label{fig:sah}
    \includegraphics[width=8cm]{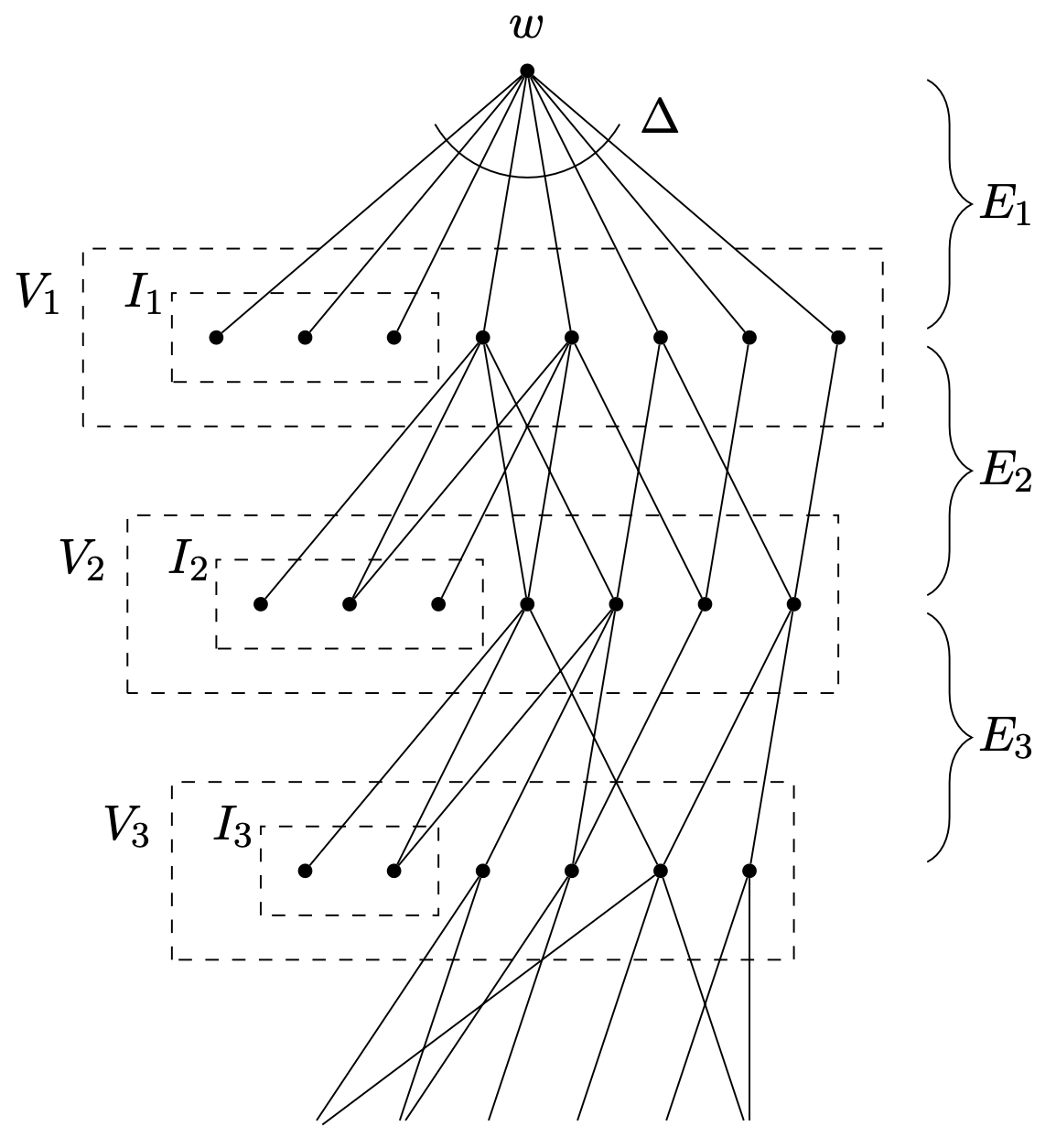}
    \caption{Setup for proof of Theorem 
    \ref{thm:sah-main}}
\end{figure}
With this setup, we return to \ref{eq:j-inequality}. 
For each $u \in V_k$, we let $d_u^+$ denote the number of neighbors of $u$ in $V_{k+1}$. 
We have 
\begin{align*}
j(G) &= 2^{\iso(G)} \prod_{(u,v) \in E} i(K_{d_u, d_v}), \\
j(G - w) &= 2^{\iso(G-w)} \prod_{\substack{(u,v) \in E_2 \\ v \in V_1}} i(K_{d_u, d_v})^{1/(d_u d_v)} \prod_{(u,v) \in E_{\geq 3}} i(K_{d_u, d_v})^{1/(d_u d_v)}, \\
j(G - w - N(w)) &= 2^{\iso(G-w-N(w))} \prod_{\substack{(u,v) \in E_3 \\ v \in V_2}} i(K_{d_u^+, d_v})^{1/(d_u d_v)} \prod_{(u,v) \in E_{\geq 4}} i(K_{d_u, d_v})^{1/(d_u d_v)}.
\end{align*}
The first equation is simply the definition of $j(G)$. 
For the second equation, each edge $(u, v) \in E$ is either in $E_2$ with $v \in V_1$ or in $E_{\ge 3}$. 
The same logic is used for the third equation. 
Thus, when considering (\ref{eq:j-inequality}), we have a shared factor of $\prod_{(u, v) \in E_{\ge 3}} i(K_{d_u, d_v})^{1/(d_ud_v)}$. 
The first two explicitly contain this term, while removing the term from the third will give a nicer expression. 
Removing this factor from the first term, we are left with the product over all edges $(u, v) \in E_{\le 2}$. 
Removing this factor from the second term, we have the product over all edges $(u, v) \in E_2$.
Finally, the third term becomes the product over all edges $(u, v) \in E_3$ divided by a factor that will be written out.
Considering the factors of the form $2^{\iso(\cdot)}$, we have 
\begin{align*}
    \iso(G) &= 0 \\
    \iso(G - w) &= |I_1| \\
    \iso(G-w-N(w)) &= |I_2|.
\end{align*}
This is evident from the diagram. 
Furthermore, we are using the fact that $G$ is connected. 
Removing this shared factor and replacing the terms of the form $2^{\iso(\cdot)}$, we are left with the equation 
\begin{align*}
    2^{|I_1|} \prod_{(u, v) \in E_2} i(K_{d_u, d_v})^{\frac{1}{d_u(d_v-1)}} + 2^{|I_2|} &\prod_{(u, v) \in E_3} \frac{i(K_{d_u^+, d_v})^{\frac{1}{d_u^+ d_v}}}{i(K_{d_u, d_v})^{\frac{1}{ d_ud_v}}} \\
    &\le \prod_{(u, v) \in E_1 \cup E_2} i(K_{d_u, d_v})^{\frac{1}{d_u d_v}} 
\end{align*}
At this point, we rewrite the expressions using the fact that 
\[i(K_{d_u, d_v}) = 2^{d_u} + 2^{d_v} - 1.\]
This is equivalent to 
\begin{align*}
2^{|I_1|}\prod_{(u,v)\in E_2}\left(2^{d_u}+2^{d_v-1}-1\right)^{\frac{1}{d_u(d_v-1)}} &+ 2^{|I_2|}\prod_{(u,v)\in E_3}\left(\frac{(2^{d_u^+}+2^{d_v}-1)^{\frac{1}{d_u^+d_v}}}{(2^{d_u}+2^{d_v}-1)^{\frac{1}{d_ud_v}}}\right)
 \\
 &\le \prod_{(u,v)\in E_1\cup E_2}\left(2^{d_u}+2^{d_v}-1\right)^{\frac{1}{d_ud_v}}
\end{align*}
This problem is significantly simpler --- we only have to points in $G$ that are at most distance $3$ away from $w$. 
Furthermore, since the graph is bipartite, there are many properties with which we can characterize these points distance $3$ away. 
The rest of the proof involves two applications of Hölder's inequality as well as careful computation.
Since this computation does not reveal too much insight on the problem and the proof strategy, the details are omitted.

\section{Further Remarks}

\label{further}

\subsection{Generalized Results}
\label{subsec:general}

In Section~\ref{Sah}, we proved an upper bound on the number of independent sets $i(G)$ for a general graph. 
In \cite{sah2019}, Sah et al. also discuss a lower bound on the number of independent sets $i(G)$. 

\begin{thm}
    \label{thm:sah-lower}
    Let $G$ be a graph and $d_u$ be the degree of vertex $v$ in $G$. 
    Then, we have 
    \[i(G) \ge \prod_{v \in V(G)} (d_v + 2)^{1/(d_v+1)},\]
    where equality holds if and only if $G$ is a disjoint union of cliques.
\end{thm}

This builds on the previous literature from Cutler and Radcliffe \cite{cutler2013maximum}.
In particular, Cutler and Radcliffe provide the same theorem for all $d$-regular graphs. 
The proof strategy is very similar to the one discussed in Section~\ref{Sah}.
In particular, we divide into sets $V_0, V_1, V_2, \dots$ again and use the same recursion. 
The key difference lies in the computation afterwards -- the proof of the lower bound does not use any techniques more than computation and algebra.
Furthermore, Sah et al. generalize this bound to one involving the partition function $P_G(\lambda) = \sum_{I \in \mathcal{I}(G)} \lambda^{|I|}$. 
In particular, we have 
\begin{thm}
    \label{thm:sah-lower-general}
    Let $G$ be a graph.
    Let $d_v$ be the degree of vertex $v$ in $G$. 
    Let $\lambda > 0$. 
    Then 
    \[P_G(\lambda) \ge \prod_{v \in V(G)} ((d_v+1)\lambda + 1)^{1/(d_v+1)},\]
    and equality holds if and only if $G$ is a disjoint union of cliques.
\end{thm}

This is a lower bound on the partition function of $G$, and it is a generalization on the lower bound on the number of independent sets. 
Setting $\lambda = 1$, we arrive at Theorem~\ref{thm:sah-lower}. Beyond the lower and upper bound on the number of independent sets $i(G)$, Sah et al. extend the result to the independent set polynomial.
\begin{definition}
    A \emph{bigraph} $G = (A, B, E)$ is a bipartite graph with bipartition $V(G) = A \sqcup B$ and edge set $E \subseteq A \times B$. 
    Then, the two-variable independent set polynomial of $G$ is 
    \[P_G(\lambda, \mu) = \sum_{I \in \mathcal{I}(G)} \lambda^{|I \cap A|} \mu^{|I \cap B|}.\]
\end{definition}
Note that setting $\mu = 1$ and letting $A = G$, we get the \emph{hard-core model} with parameter $\lambda$ discussed in class:
\[P_G(\lambda, 1) = \sum_{I \in \mathcal{I}(G)} \lambda^{|I|}.\]
Furthermore, setting $\lambda = 1$ and $\mu = 1$, we get the number of independent sets $i(G)$. 
Sah et al. prove the result 
\begin{thm}
    \label{thm:sah-general}
    Let $G = (A, B, E)$ be a bigraph without isolated vertices. 
    Let $d_v$ denote the degree of $v$ in $G$. 
    Let $\lambda, \mu > 0$. 
    Then 
    \[P_G(\lambda, \mu) \leq \prod_{\substack{uv \in E \\ u \in A, v \in B}} \left( (1 + \lambda)^{d_v} + (1 + \mu)^{d_u} - 1 \right)^{1 / (d_u d_v)},\]
    and equality holds if and only if $G$ is a disjoint union of complete bipartite graphs. 
\end{thm}
First, note that setting $\lambda = \mu = 1$ gives Theorem~\ref{thm:sah-main}.
Furthermore, we can easily extend this result to all bigraphs by adding a factor of $2^{\iso(G)}$. 
The proof technique is exactly the same as the one described in Section~\ref{Sah}.





\bibliographystyle{plain}
\bibliography{main}

\begin{thebibliography}{1}

\bibitem{alon1991independent}
Noga Alon.
\newblock Independent sets in regular graphs and sum-free subsets of finite groups.
\newblock {\em Israel journal of mathematics}, 73(2):247--256, 1991.

\bibitem{banff}
Peter~J. Cameron and Paul Erd\H{o}s.
\newblock On the number of sets of integers with various properties.
\newblock {\em Number Theory (Banff, AB, 1988), de Gruyter, Berlin}, pages 61--79, 1990.

\bibitem{cutler2013maximum}
Jonathan Cutler and A.~J. Radcliffe.
\newblock The maximum number of complete subgraphs in a graph with given maximum degree, 2013.

\bibitem{galvin2009upper}
David Galvin.
\newblock An upper bound for the number of independent sets in regular graphs.
\newblock {\em Discrete mathematics}, 309(23-24):6635--6640, 2009.

\bibitem{KAHN_2001}
Jeff Kahn.
\newblock An entropy approach to the hard-core model on bipartite graphs.
\newblock {\em Combinatorics, Probability and Computing}, 10(3):219–237, 2001.

\bibitem{sah2019}
Ashwin Sah, Mehtaab Sawhney, David Stoner, and Yufei Zhao.
\newblock The number of independent sets in an irregular graph.
\newblock {\em Journal of Combinatorial Theory, Series B}, 138:172--195, 2019.

\bibitem{ZHAO_2009}
Yufei Zhao.
\newblock The number of independent sets in a regular graph.
\newblock {\em Combinatorics, Probability and Computing}, 19(2):315–320, November 2009.

\bibitem{zhao2011bipartite}
Yufei Zhao.
\newblock The bipartite swapping trick on graph homomorphisms.
\newblock {\em SIAM Journal on Discrete Mathematics}, 25(2):660--680, 2011.

\end{thebibliography}

\end{document}